\documentclass[12pt,leqno,openbib]{amsart}
\usepackage{amssymb,amsmath,graphicx,amsfonts,euscript}

\newtheorem{thm}{Theorem}[section]

\newtheorem{prop}[thm]{Proposition}
\newtheorem{define}[thm]{Definition}
\newtheorem{rem}[thm]{Remark}

\newtheorem{assume}[thm]{Assumption}

\newcommand{\bb}{\begin{equation}}
\newcommand{\ee}{\end{equation}}
\newcommand{\bq}{\begin{eqnarray}}
\newcommand{\eq}{\end{eqnarray}}
\newcommand{\bqn}{\begin{eqnarray*}}
\newcommand{\eqn}{\end{eqnarray*}}
\def\pp{\partial}

\numberwithin{equation}{section}

\begin{document}
\title[Generalized surface quasi-geostrophic equations]
{Inviscid models generalizing the 2D Euler and the surface quasi-geostrophic equations}
\author[Dongho Chae, Peter Constantin and Jiahong Wu]{Dongho Chae$^{1}$, Peter Constantin$^{2}$ and Jiahong Wu$^{3}$}
\address{$^1$ Department of Mathematics,
Sungkyunkwan University,
Suwon 440-746, Korea}

\address{$^2$Department of Mathematics,
 University of Chicago,
 5734 S. University Avenue,
Chicago, IL 60637, USA.}

\address{$^3$Department of Mathematics,
Oklahoma State University,
401 Mathematical Sciences,
Stillwater, OK 74078, USA.}

\email{chae@skku.edu}
\email{const@cs.uchicago.edu}
\email{jiahong@math.okstate.edu}

\begin{abstract}
Any classical solution of the 2D incompressible Euler equation is global in time.
However, it remains an outstanding open problem whether classical solutions of the
surface quasi-geostrophic (SQG) equation preserve their regularity for all time.
This paper studies solutions of a family of active scalar equations in
which each component $u_j$ of the velocity field $u$ is determined by the scalar $\theta$ through $u_j =\mathcal{R} \Lambda^{-1} P(\Lambda) \theta$ where $\mathcal{R}$ is a Riesz transform and $\Lambda=(-\Delta)^{1/2}$. The 2D Euler vorticity equation corresponds to the special case $P(\Lambda)=I$ while the SQG equation to the case $P(\Lambda) =\Lambda$.  We develop tools to bound $\|\nabla u||_{L^\infty}$ for a general class of operators $P$ and establish the global regularity for the Loglog-Euler equation for which $P(\Lambda)= (\log(I+\log(I-\Delta)))^\gamma$ with $0\le \gamma\le 1$. In addition, a regularity criterion for the model corresponding to $P(\Lambda)=\Lambda^\beta$ with $0\le \beta\le 1$ is also obtained.
\end{abstract}

\maketitle

\section{Introduction and statements of the main results}
\label{intr}
\setcounter{equation}{0}

This paper studies solutions of the active scalar equation
\begin{equation} \label{general}
\left\{
\begin{array}{l}
\pp_t \theta + u\cdot\nabla \theta =0, \quad x\in \mathbb{R}^d, \, t>0, \\
u = (u_j), \quad u_j = \mathcal{R}_l \Lambda^{-1} P(\Lambda)\, \theta,\quad 1\le j, \,l\le d,
\end{array}
\right.
\end{equation}
where $\theta =\theta(x,t)$ is a scalar function of $x\in \mathbb{R}^{d}$ and $t\ge 0$, $u$ denotes a velocity field with its component $u_j$ ($1\le j\le d$) given by a Riesz transform $\mathcal{R}_l$ applied to $\Lambda^{-1} P(\Lambda)\, \theta$.
Here the operators $\Lambda = (-\Delta)^{\frac12}$, $P(\Lambda)$ and $\mathcal{R}_l$ are defined through their Fourier transforms, namely
$$
\widehat{\Lambda f}(\xi) = |\xi| \widehat{f}(\xi), \quad \widehat{P(\Lambda) f}(\xi) = P(|\xi|) \widehat{f}(\xi), \quad \widehat{\mathcal{R}_l f}(\xi)= \frac{i\,\xi_l}{|\xi|}\, \widehat{f}(\xi),
$$
where $1\le l\le d$ is an integer, $\widehat{f}$ or $\mathcal{F}(f)$ denotes the Fourier transform,
$$
\widehat{f}(\xi) = \mathcal{F}(f)(\xi) =\frac{1}{(2\pi)^{d/2}} \int_{\mathbb{R}^d} e^{-i x\cdot \xi} f(x)\,dx.
$$

\vskip .1in
Our consideration is restricted to $P$ satisfying the following Assumption.
\begin{assume} \label{P_con}
The symbol $P=P(|\xi|)$ assumes the following properties:
\begin{enumerate}
\item $P$ is continuous on $\mathbb{R}^d$ and $P\in C^\infty(\mathbb{R}^d\setminus\{0\})$;
\item $P$ is radially symmetric;
\item $P=P(|\xi|)$ is nondecreasing in $|\xi|$;
\item There exist two constants $C$ and $C_0$ such that
\begin{equation*}
\sup_{2^{-1} \le |\eta| \le 2}\, \left|(I-\Delta_\eta)^n \,P(2^j |\eta|)\right| \le C\, P(C_0 \, 2^j)
\end{equation*}
for any integer $j$ and $n=1,2,\cdots, 1+ \left[\frac{d}{2}\right]$.
\end{enumerate}
\end{assume}
We remark that (4) in Assumption \ref{P_con} is a very natural condition on symbols of Fourier multiplier operators and is similar to the main condition in the Mihlin-H\"{o}rmander Multiplier Theorem (see e.g. \cite[p.96]{St}). For notational convenience, we also assume that $P\ge 0$. Some special examples of $P$ are
\begin{eqnarray*}
&& P(\xi) = \left(\log(1 +|\xi|^2)\right)^\gamma \quad\mbox{with $\gamma\ge 0$}, \\
&& P(\xi) = \left(\log(1+\log(1 +|\xi|^2))\right)^\gamma \quad\mbox{with $\gamma\ge 0$}, \\
&& P(\xi) = |\xi|^\beta \quad\mbox{with $\beta\ge 0$},\\
&& P(\xi) = (\log(1 +|\xi|^2))^\gamma\,|\xi|^\beta \quad\mbox{with $\gamma\ge 0$ and $\beta\ge 0$}.
\end{eqnarray*}

\vskip .1in
A particularly important case of (\ref{general}) is the 2D active scalar equation
\begin{equation} \label{general2d}
\left\{
\begin{array}{l}
\pp_t \theta + u\cdot\nabla \theta =0, \quad x\in \mathbb{R}^2, \, t>0, \\
u = \nabla^\perp \psi\equiv (-\pp_{x_2}\psi, \pp_{x_1} \psi), \quad
-\Lambda^2 \psi = P(\Lambda)\, \theta
\end{array}
\right.
\end{equation}
which generalizes the 2D Euler vorticity equation
\begin{equation}\label{euler}
\left\{
\begin{array}{l}
\pp_t \omega + u \cdot \nabla \omega =0,\\
u =\nabla^\perp \psi, \quad \Delta\psi=\omega
\end{array}
\right.
\end{equation}
and the surface quasi-geostrohic (SQG) equation
\begin{equation}\label{SQG}
\left\{
\begin{array}{l}
\pp_t \theta + u \cdot \nabla \theta = 0,\\u=\nabla^\perp \psi, \quad -\Lambda\psi = \theta.
\end{array}
\right.
\end{equation}
The 2D Euler equation has been extensively studied and its global regularity has long been established (see e.g. \cite{Che}, \cite{MaBe} and \cite{MaPu}). The SQG equation and its dissipative counterpart have recently attracted a lot of attention and numerous efforts have been devoted to the global regularity and related issues concerning their solutions  (see e.g. \cite{AbHm}, \cite{Bae}, \cite{Bar}, \cite{Blu}, \cite{CaS}, \cite{CV}, \cite{CaFe}, \cite{Ch}, \cite{ChJDE}, \cite{Cha}, \cite{Cha2}, \cite{Cha4}, \cite{CCCF}, \cite{ChL}, \cite{CMZ1}, \cite{Chen}, \cite{Con}, \cite{CCW}, \cite{CIW}, \cite{CLS}, \cite{CMT}, \cite{CNS}, \cite{CWnew1}, \cite{CWnew2}, \cite{Cor}, \cite{CC}, \cite{CoFe1}, \cite{CoFe2}, \cite{CoFe3}, \cite{CFMR}, \cite{Dab}, \cite{DHLY}, \cite{DoCh}, \cite{Dong}, \cite{DoDu}, \cite{DoLi0}, \cite{DoLi}, \cite{DoPo}, \cite{DoPo2}, \cite{FPV}, \cite{FrVi}, \cite{Gil}, \cite{HPGS}, \cite{HmKe}, \cite{HmKe2}, \cite{Ju}, \cite{Ju2}, \cite{KhTi}, \cite{Ki1}, \cite{Ki2}, \cite{KN1}, \cite{KN2}, \cite{KNV}, \cite{Li}, \cite{LiRo}, \cite{Maj}, \cite{MaBe}, \cite{MaTa}, \cite{Mar1}, \cite{Mar2}, \cite{Mar3}, \cite{MarLr}, \cite{May}, \cite{MayZ}, \cite{MiXu}, \cite{Mi}, \cite{NiSc}, \cite{OhYa1}, \cite{Pe}, \cite{ReDr}, \cite{Res}, \cite{Ro1}, \cite{Ro2}, \cite{Sch}, \cite{Sch2}, \cite{Si}, \cite{Si2}, \cite{Sta}, \cite{WaJi}, \cite{Wu97}, \cite{Wu2}, \cite{Wu01}, \cite{Wu02}, \cite{Wu3}, \cite{Wu4}, \cite{Wu4}, \cite{Wu41}, \cite{Wu31}, \cite{Wu77}, \cite{Yu}, \cite{Yuan}, \cite{YuanJ}, \cite{Zha0}, \cite{Zha}, \cite{Zhou}, \cite{Zhou2}).

\vskip .1in
\vskip .1in
The goal of this paper is to understand the global regularity issue concerning solutions of (\ref{general}) with a given initial datum
\begin{equation} \label{IC}
\theta(x,0) = \theta_0(x), \quad  x\in \mathbb{R}^d.
\end{equation}
The key quantity involved in this issue is $\|\nabla u\|_{L^\infty}$. Tools are developed here to bound $\|\Delta_j \nabla u\|_{L^p} $ and $\|S_N \nabla u\|_{L^p} $ when a vector field $u: \mathbb{R}^d\to \mathbb{R}^d$ is related to a scalar function $\theta$ by
$$
(\nabla u)_{jk}  = \mathcal{R}_l \mathcal{R}_m\, P(\Lambda) \, \theta
$$
where $1\le j,k,l,m\le d$, $(\nabla u)_{jk}$ denotes the $(j,k)$-th entry of $\nabla u$ and $\mathcal{R}_l$ and $\mathcal{R}_m$ denote the Riesz transforms. Here $\Delta_j$ with $j\ge -1$ denotes the Fourier localization operator and
$$
S_N = \sum_{j=-1}^{N-1} \Delta_j.
$$
The precise definitions of $\Delta_j$ and $S_N$ are provided in Appendix \ref{Besov}.
The assumption that $u$ is divergence-free is not used in deriving these bounds. The bounds obtained here are summarized in the following theorem.

\begin{thm} \label{nablau}
Let $u: \mathbb{R}^d\to \mathbb{R}^d$ be a vector field. Assume that $u$ is related to a scalar $\theta$ by
$$
(\nabla u)_{jk}  = \mathcal{R}_l \mathcal{R}_m\, P(\Lambda) \, \theta,
$$
where $1\le j,k,l,m\le d$, $(\nabla u)_{jk}$ denotes the $(j,k)$-th entry of $\nabla u$, $\mathcal{R}_l$ denotes the Riesz transform, and $P$ obeys Assumption \ref{P_con}. Then, for any integers $j\ge 0$ and $N\ge 0$,
\begin{eqnarray}
\|S_N \nabla u\|_{L^p} &\le& C_{p,d}\, P(C_0 2^N)\,\|S_{N} \theta\|_{L^p}, \quad 1<p<\infty,
\label{bound1} \\
\|\Delta_j \nabla u\|_{L^q} &\le& C_d\,  P(C_0 2^j)\,\|\Delta_j \theta\|_{L^q}, \quad 1\le q\le \infty,
\label{bound2} \\
\|S_N \nabla u\|_{L^\infty} &\le&  C_d\,\|\theta\|_{L^1\cap L^\infty} + C_d\, N\,P(C_0 2^N)\,\|S_{N+1}\theta\|_{L^\infty}, \label{bound3}
\end{eqnarray}
where $C_{p,d}$ is a constant depending on $p$ and $d$ only and $C_d$s' depend on $d$ only.
\end{thm}

\vskip .1in
We remark that in general the constant $C_{p,d}$ grows linearly with respect to $p$ and thus (\ref{bound1}) does not follow for $p=\infty$. With these bounds at our disposal, we are able to establish global regularity results covering two special cases of $P$. The first result is for (\ref{general}) with $P(|\xi|) = \left(\log(1+\log(1 +|\xi|^2))\right)^\gamma$. For the simplicity of our presentation here, we state the result for the 2D case of (\ref{general}), namely
\begin{equation} \label{loglog-euler}
\left\{
\begin{array}{l}
\pp_t \theta + u\cdot\nabla \theta =0, \\
u = \nabla^\perp \psi, \quad
\Delta \psi = \left(\log(1+\log(1 -\Delta))\right)^\gamma\, \theta,
\end{array}
\right.
\end{equation}
which we call the Loglog-Euler equation. Although any classical solution $\theta$ of (\ref{loglog-euler}) obeys the global {\it a priori} bound
$$
\|\theta(\cdot,t)\|_{L^p} \le \|\theta(\cdot,0)\|_{L^p} \quad\mbox{for any}\quad1\le p \le \infty,
$$
the regularity of the velocity $u$ recovered from the relation
$$
u = \nabla^\perp \psi, \quad
\Delta \psi = \left(\log(1+\log(1 -\Delta))\right)^\gamma\, \theta
$$
is worse than in the case of the 2D Euler equation.  Nevertheless we are able to obtain the global regularity for (\ref{loglog-euler}) with $0\le \gamma\le 1$.

\begin{thm} \label{loglog-global}
Consider the initial-value problem (\ref{loglog-euler}) and (\ref{IC}) with $\gamma$ and $\theta_0$ satisfying
\begin{equation}\label{theta0}
0\le \gamma\le 1, \qquad \theta_0 \in L^1(\mathbb{R}^2) \cap L^\infty (\mathbb{R}^2) \cap B^s_{q,\infty} (\mathbb{R}^2)
\end{equation}
where $2< q \le \infty$ and $s>1$.  Then the initial-value problem (\ref{loglog-euler}) and (\ref{IC}) has a unique global solution $\theta$ satisfying,
$$
\theta \in L^\infty([0,\infty); B^{s}_{q,\infty}(\mathbb{R}^2)), \quad \nabla u \in L^\infty([0,\infty); B^{1+ s_1}_{q,\infty}(\mathbb{R}^2)),
$$
where $s_1<s$.
\end{thm}

The general version of Theorem \ref{loglog-global}, namely the global regularity result for (\ref{general}) will be stated in Section \ref{APbd}. Here $B^s_{q,\infty}$ denotes an inhomogeneous Besov space. The definition of a general Besov space $B^s_{p,q}$ is provided in Appendix \ref{Besov}. Even though $\theta_0\in B^s_{q,\infty}$ implies $\theta_0\in L^\infty$, the condition on $\theta_0$ is written as in (\ref{theta0}) to emphasize the importance of $L^\infty$ assumption. The global regularity stated in the Besov space setting in Theorem \ref{loglog-global} can be converted into a global regularity statement in Sobolev spaces. Combining Theorem \ref{loglog-global} and the embedding relations
$$
W^r_q \hookrightarrow  B^r_{q,\infty}\hookrightarrow B^{r_1}_{q,\min\{2,q\}} \hookrightarrow W^{r_1}_q, \quad r>r_1,
$$
we can conclude that any initial data in $W^r_q$ with $2< q \le \infty$ and $r>1$ would yield a global solution in $W^{r_1}_q$ for any $r_1<r$.

\vskip .1in
Theorem \ref{loglog-global} is proven by combining the Besov space techniques and  the following extrapolation inequality.

\begin{prop} \label{LogSob}
Let $u: \mathbb{R}^d\to \mathbb{R}^d$ be a vector field. Assume that $u$ is related to a scalar $\theta$ by
\begin{equation}\label{ujk}
(\nabla u)_{jk}  = \mathcal{R}_l \mathcal{R}_m\, \left(\log(I+ \log(I-\Delta))\right)^\gamma \, \theta
\end{equation}
where $\gamma\ge 0$, $1\le j,k,l,m\le d$, $(\nabla u)_{jk}$ denotes the $(j,k)$-th entry of $\nabla u$ and $\mathcal{R}_l$ and $\mathcal{R}_m$ denote the Riesz transforms. Then, for any $1\le q\le \infty$ and $s>d/q$,
$$
\|\nabla u\|_{L^\infty} \le \|\theta\|_{L^1\cap L^\infty} + C\,\|\theta\|_{L^\infty}\, \log(1+\|\theta\|_{B^s_{q,\infty}})\,
\left(\log\left(1+\log(1+\|\theta\|_{B^s_{q,\infty}})\right)\right)^\gamma
$$
where $C$ is a constant that depends on $d$, $q$ and $s$ only.
\end{prop}

\vskip .1in
The second special case studied here is when $P(|\xi|) = |\xi|^\beta$ with $0\le \beta \le 1$. Our aim is to understand how the parameter $\beta$ affects the regularity of solutions to the initial-value problem
\begin{equation} \label{gsqg}
\left\{
\begin{array}{l}
\pp_t \theta + u\cdot\nabla \theta =0 \\
u = \nabla^\perp \psi,\quad -\Lambda^{2} \psi = \Lambda^\beta \theta,
\end{array}
\right.
\end{equation}
where $0\le \beta\le 1$. The evolution of patch-like initial data under (\ref{gsqg}) has previously been studied in \cite{CFMR}. Clearly (\ref{gsqg}) bridges the 2D Euler and the SQG equation. It is hoped that this study would shed light on the global regularity issue concerning the SQG equation.

\vskip .1in
It is unknown if all classical solutions of (\ref{gsqg}) conserve their regularity for all time except in the case of the 2D Euler equation. In order to deal with the global regularity for (\ref{gsqg}), it suffices to obtain a suitable bound for $\|\nabla u\|_{L^\infty(\mathbb{R}^2)}$. Intuitively, the relation
$$
u= -\nabla^\perp \Lambda^{-2+\beta} \theta
$$
implies that $\|\nabla u\|_{L^\infty(\mathbb{R}^2)}$ can be bounded more or less by  a bound for $\Lambda^\beta \theta$. In fact, this intuitive idea can be made rigorous and is reflected in the following logarithmic H\"{o}lder inequality
$$
\|S\|_{L^\infty} \le C \|\theta\|_{C^\beta} \,\ln (1+ \|\theta\|_{C^\sigma}) + C\, \|\theta\|_{L^q},\quad \sigma >\beta, \,\,q>1,
$$
where $S$ denotes the symmetric part of $\nabla u$ and $C^\beta$ the H\"{o}lder space. This inequality, together with a bound for the back-to-labels map determined by $u$, allows us to obtain the following regularity criterion.
\begin{thm}\label{crit10}
Consider (\ref{gsqg}) with $0\le \beta\le 1$. Let $\theta$ be a solution of (\ref{gsqg}) corresponding to the data $\theta_0 \in C^\sigma(\mathbb{R}^2)\cap L^q(\mathbb{R}^2)$ with $\sigma>1$ and $q>1$. Let $T>0$. If $\theta$ satisfies
$$
\int_0^{T}\|\theta(\cdot, t)\|_{C^\beta(\mathbb{R}^2)} \,dt < \infty,
$$
then $\theta$ remains in $C^\sigma(\mathbb{R}^2)\cap L^q(\mathbb{R}^2)$ on the time interval $[0,T]$.
\end{thm}
This criterion especially establishes the global regularity for the 2D Euler equation and reduces to the well-known criterion for the SQG equation when $\beta=1$ (see \cite{CMT}).

\vskip .1in
The rest of this paper is organized as follows. Section \ref{nablaubd} is devoted to the bounds in Theorem \ref{nablau} and Proposition \ref{LogSob}. Theorem \ref{loglog-global} and its general version, the global regularity result for (\ref{general}) are stated and proven in  Section \ref{APbd}. Section \ref{GISQG} details the proof of Theorem \ref{crit10}. Appendix \ref{Besov} provides the definition of Besov spaces and some related facts.

\vskip .3in
\section{Bounds for $\|\Delta_j \nabla u\|_{L^q}$, $\|S_N \nabla u\|_{L^q}$ and $\|\nabla u\|_{L^\infty}$}
\label{nablaubd}

This section derives the bounds stated in Theorem \ref{nablau} and proves the logarithmic interpolation inequality presented in Proposition \ref{LogSob}.

\vskip .1in
We make use of  Mihlin and H\"{o}rmander Multiplier
Theorem (see \cite[p.96]{St}) in the proof of (\ref{bound1}). This theorem is recalled first.

\begin{thm}\label{MH}
Suppose that $Q(\xi)$ is of class $C^k$ in the complement of the origin of $\mathbb{R}^d$, where $k>\frac{d}{2}$ is an integer. Assume also that
\begin{equation}\label{MH_con}
\left|D^\alpha Q(\xi)\right| \le B\, |\xi|^{-|\alpha|}, \quad \mbox{whenever $|\alpha|\le k$}.
\end{equation}
Then $Q\in \mathcal{M}_q$, $1<q<\infty$. That is, $\|T_Q f\|_{L^q} \le C_q \,\|f\|_{L^q}$, where $T_Q$ is defined by
$$
\widehat{T_Q f} (\xi) = Q(\xi)\, \widehat{f}(\xi).
$$
\end{thm}

\vskip .1in
For further reference, we rewrite  (\ref{bound1})
as a proposition.
\begin{prop} \label{bound1p}
Let $u: \mathbb{R}^d\to \mathbb{R}^d$ be a vector field. Assume that $u$ is related to a scalar $\theta$ by
\begin{equation}\label{ujkp}
(\nabla u)_{jk}  = \mathcal{R}_l \mathcal{R}_m\, P(\Lambda) \, \theta
\end{equation}
where $1\le j,k,l,m\le d$, $(\nabla u)_{jk}$ denotes the $(j,k)$-th entry of $\nabla u$, $\mathcal{R}_l$ denotes the Riesz transform and $P$ obeys Assumption \ref{P_con}. Then, for any integer $N\ge 0$,
\begin{equation} \label{snnup}
\|S_N \nabla u\|_{L^p} \le C_{p,d}\, P(C_0\,2^N)\,\|S_N \theta\|_{L^p}, \quad 1<p<\infty,
\end{equation}
where $C_{p,d}$ is a constant depending on $p$ and $d$ only.
\end{prop}

\begin{proof}
As detailed in Appendix \ref{Besov}, the symbol of $S_N$ is $\psi(\xi/2^N)$ with $\psi$ satisfying
$$
\psi  \in C_0^\infty(\mathbb{R}^d), \quad \mbox{supp} \psi \subset B\left(0,\frac{11}{12}\right), \quad \psi(\xi)=1 \,\,\mbox{for}\,\, |\xi|\le \frac34.
$$
It follows from (\ref{ujkp}) that
$$
\widehat{(S_N \nabla u)}_{jk}(\xi) = Q(\xi) \,P(C_0 2^N)\,\widehat{S_N \theta} (\xi)
$$
where $Q(\xi)$ is supported on $|\xi|\le (11/12) 2^N$ and,
for $|\xi| \le (11/12) 2^N$,
$$
Q(\xi)= - \frac{\xi_l\,\xi_m}{|\xi|^2} \, \frac{P(|\xi|)}{P(C_0 2^N)}.
$$
To apply Theorem \ref{MH}, we verify (\ref{MH_con}). Clearly, for any $\alpha$ with $|\alpha|=0,1, \cdots, 1+\left[\frac{d}{2}\right]$,
$$
\left|D^\alpha \frac{\xi_l\,\xi_m}{|\xi|^2} \right| \le C \, |\xi|^{-|\alpha|}.
$$
In addition, for any $\xi\not =0$, there is an integer $j$ such that  $\xi = 2^{j} \eta$ with $2^{-1} \le |\eta| \le 2$. Trivially, for $\xi$ in the support of $Q$, $j\le N$. It is easy to see that Condition (4) in Assumption \ref{P_con} implies that
$$
\sup_{2^{-1} \le |\eta| \le 2}\, \left|(-\Delta_\eta)^n \,P(2^j |\eta|)\right| \le C\, P(C_0 \, 2^j)
$$
for $n=0,1,\cdots, 1+\left[\frac{d}{2}\right]$. Then,
\begin{eqnarray*}
\left|(-\Delta_\xi)^n \frac{P(|\xi|)}{P(C_0 2^N)} \right| &=& \left|(-\Delta_\eta)^n \frac{2^{-2nj}\,P(2^j |\eta|)}{P(C_0 2^N)} \right| \\
&\le& |\eta|^{2n}\, |2^{j} \eta|^{-2n} \,\frac{P(C_02^j)}{P(C_0 2^N)}\\
&\le& |\eta|^{2n}\, |\xi|^{-2n}.
\end{eqnarray*}
This verifies (\ref{MH_con}). (\ref{snnup}) then follows as a consequence of Theorem \ref{MH}.
\end{proof}

\vskip .1in
For the sake of clarity, we restate (\ref{bound2}) in Theorem \ref{nablau} as a proposition.
\begin{prop} \label{important}
Let $u: \mathbb{R}^d\to \mathbb{R}^d$ be a vector field. Assume that $u$ is related to a scalar $\theta$ by
$$
(\nabla u)_{jk}  = \mathcal{R}_l \mathcal{R}_m\, P(\Lambda) \, \theta
$$
where $1\le j,k,l,m\le d$, $(\nabla u)_{jk}$ denotes the $(j,k)$-th entry of $\nabla u$ and $\mathcal{R}_l$ denotes the Riesz transform. Here $P$ obeys Assumption \ref{P_con}. Then, for any integer $N\ge 0$,
\begin{equation}\label{m1desired}
\|\Delta_N \nabla u\|_{L^q} \le  C_d\,  P(C_0 2^N)\,\|\Delta_N \theta\|_{L^q}, \quad 1\le q\le \infty.
\end{equation}
where $C_d$ is a constant depending on $d$ only.
\end{prop}

\begin{rem}
This proposition is invalid in the case when $N=-1$. The proof requires the symbol of $\Delta_N$ is supported away from the origin.
\end{rem}

\begin{proof}[Proof of Proposition \ref{important}]
Clearly,
$$
\left(\Delta_N \nabla u\right)_{jk} =  \mathcal{R}_l \mathcal{R}_m\,P(\Lambda) \Delta_N \theta
$$
and
$$
\widehat{(\Delta_N \nabla u)}_{jk}(\xi) = - \frac{\xi_l\,\xi_m}{|\xi|^2} \,P(|\xi|)
\, \widehat{\Delta_N \theta}(\xi).
$$
As defined in Appendix \ref{Besov}, $\widehat{\Delta_N \theta}(\xi) =\phi(\xi/2^N)\, \widehat{\theta}(\xi)$ with $\phi(\xi/2^N)$ supported in the annulus $(3/4) 2^N \le |\xi| \le (11/6) 2^N$. It is not hard to see that there exists a smooth radial function  $\widetilde{\phi}_N$ satisfying
$$
\widetilde{\phi}_N \equiv 1\,\,\mbox{for}\,\,(3/4) 2^N \le |\xi| \le (11/6) 2^N \quad\mbox{and}\quad  \mbox{supp}\,\widetilde{\phi}_N \subset\{\xi: 2^{N-1} \le |\xi| \le 2^{N+1}\}.
$$
Then
$$
\widehat{(\Delta_N \nabla u)}_{jk}(\xi) = - \frac{\xi_l\,\xi_m}{|\xi|^2} \,P(|\xi|)
\,\widetilde{\phi}_N(\xi)\, \widehat{\Delta_N \theta}(\xi)
$$
or
$$
(\Delta_N \nabla u)_{jk} = g\ast \Delta_N \theta,
$$
where $g$ denotes the inverse Fourier transform
$$
g(x)= \frac{1}{(2\pi)^{d/2}}\, \int_{\mathbb{R}^d} e^{ix\cdot \xi} \,\left(- \frac{\xi_l\,\xi_m}{|\xi|^2} \,P(|\xi|)
\,\widetilde{\phi}_N(\xi) \right) \, d\xi
$$
Clearly, $g(x) = 2^{Nd} g_1(2^N x)$, where
$$
g_1(x) = - \frac{1}{(2\pi)^{d/2}}\, \int_{2^{-1} \le |\eta|\le 2} 2^{ix\cdot\eta} \,
 \frac{\eta_l\,\eta_m}{|\eta|^2} \,P(2^N |\eta|) \,\widetilde{\phi}_0(\eta) \, d\eta
$$
with $\widetilde{\phi}_0(\eta) = \widetilde{\phi}_N(2^N \eta)$. To show $g\in L^1(\mathbb{R}^d)$, it suffices to show $g_1\in L^1(\mathbb{R}^d)$. Since
$$
(1+|x|^2)^n\, g_1(x)= -\frac{1}{(2\pi)^{d/2}}\,\int_{2^{-1} \le |\eta|\le 2} 2^{ix\cdot\eta} (I-\Delta_\eta)^n\, \frac{\eta_l\,\eta_m}{|\eta|^2} \,P(2^N |\eta|) \,\widetilde{\phi}_0(\eta)\, d\eta,
$$
we obtain, by (4) in Assumption \ref{P_con},
$$
(1+|x|^2)^n\, |g_1(x)| \le C\, P(C_0\, 2^N).
$$
where $C$ is constant independent of $N$. (\ref{m1desired}) then follows from Young's inequality.
\end{proof}

\vskip .1in
We now prove (\ref{bound3}) of Theorem \ref{nablau}. In fact, we have the following proposition.

\begin{prop}
Let $u: \mathbb{R}^d\to \mathbb{R}^d$ be a vector field. Assume that $u$ is related to a scalar $\theta$ by
$$
(\nabla u)_{jk}  = \mathcal{R}_l \mathcal{R}_m\, P(\Lambda) \, \theta,
$$
where $1\le j,k,l,m\le d$, $(\nabla u)_{jk}$ denotes the $(j,k)$-th entry of $\nabla u$ and $\mathcal{R}_l$ denotes the Riesz transform. Here $P$ obeys Assumption \ref{P_con}. Then, for any integer $N\ge 0$,
\begin{eqnarray}
\|S_N \nabla u\|_{L^\infty} &\le&  C_d\,\|\theta\|_{L^1\cap L^\infty} + C_d\, N\,P(C_0 2^N)\,\|S_{N+1}\theta\|_{L^\infty}, \label{bbb1},
\end{eqnarray}
where $C_d$ depends on $d$ only.
\end{prop}

\begin{proof}
Splitting $S_N$ into two parts and applying Proposition \ref{important} with $q=\infty$,
we have
\begin{eqnarray}
\|\nabla S_N u\|_{L^\infty} &\le& \|\nabla \Delta_{-1} u\|_{L^\infty} + \sum_{j=0}^{N-1} \|\nabla \Delta_j u\|_{L^\infty} \nonumber \\
&\le& C_d\, \|\Delta_{-1} \theta\|_{L^2} + \sum_{j=0}^{N-1} C_d\, P(C_0\, 2^j)\,\|\Delta_j \theta\|_{L^\infty} \label{hhh}
\end{eqnarray}
Since $P$ is nondecreasing according to Assumption \ref{P_con}
and the simple fact that
$$
\|\Delta_j \theta\|_{L^\infty} \le \|S_{N+1} \theta\|_{L^\infty}, \quad j=0, 1, \cdots, N-1,
$$
we have
$$
\|\nabla S_N u\|_{L^\infty} \le C_d\, \|\theta\|_{L^1\cap L^\infty} + C_d\, N P(C_0\,2^N)\,\|S_{N+1} \theta\|_{L^\infty},
$$
which is (\ref{bbb1}).
\end{proof}

\vskip .1in
We now prove Proposition \ref{LogSob}, in which $P$ assumes the special form
$$
P(\Lambda) = \left(\log(I+ \log(I-\Delta))\right)^\gamma.
$$

\begin{proof}[Proof of Proposition \ref{LogSob}]  For any integer $N\ge 0$, we have
$$
\|\nabla u\|_{L^\infty} \le \|\Delta_{-1} \nabla u\|_{L^\infty} + \sum_{k=0}^{N-1} \|\Delta_k \nabla u\|_{L^\infty} + \sum_{k=N}^{\infty}  \|\Delta_k \nabla u\|_{L^\infty}.
$$
By Bernstein's inequality and Proposition \ref{important}, we have
\begin{eqnarray*}
\|\nabla u\|_{L^\infty} &\le& C_d\,\|\theta\|_{L^1\cap L^\infty} + C_d\, N \left(\log(1+ \log(1+2^{2(N-1)}))\right)^\gamma \, \|\theta\|_{L^\infty} \\
&&  + C_d\,\sum_{k=N}^{\infty} (2^k)^{\frac{d}{q}}\, \|\nabla \Delta_k u\|_{L^q}.
\end{eqnarray*}
Since $\log (1+2^{2(N-1)}) =(\log_2 e)^{-1}\, \log_2 (1+2^{2(N-1)})\le 2N$, we apply Proposition \ref{important} again to obtain
\begin{eqnarray*}
\|\nabla u\|_{L^\infty} &\le& C_d\,\|\theta\|_{L^1\cap L^\infty} + C_d\, N \left(\log(1+N)\right)^\gamma \|\theta\|_{L^\infty} \\
&& + \,C_d\, \sum_{k=N}^{\infty} (2^k)^{\frac{d}{q}}\,\left(\log(1+ k)\right)^\gamma \|\Delta_k \theta\|_{L^q}.
\end{eqnarray*}
By the definition of Besov space $B^s_{q,\infty}$ (see Appendix \ref{Besov}),
$$
\|\Delta_k \theta\|_{L^q} \le 2^{-s\,k} \|\theta\|_{B^{s}_{q,\infty}}.
$$
Therefore,
\begin{eqnarray*}
\|\nabla u\|_{L^\infty} &\le& C_d\,\|\theta\|_{L^1\cap L^\infty} + C_d\, N \left(\log(1+N)\right)^\gamma \|\theta\|_{L^\infty}\\
&& + \,C_d\, \|\theta\|_{B^{s}_{q,\infty}} \sum_{k=N}^{\infty} (2^k)^{(\frac{d}{q}-s)}\,\left(\log(1+ k)\right)^\gamma.
\end{eqnarray*}
Since $d/q-s<0$, we obtain for large $N$,
\begin{eqnarray*}
\|\nabla u\|_{L^\infty} &\le& C_d\,\|\theta\|_{L^1\cap L^\infty} + C_d\, N \left(\log(1+N)\right)^\gamma \|\theta\|_{L^\infty} \\
&& + \,C_{d,q,s}\, \|\theta\|_{B^{s}_{q,\infty}} \,(2^N)^{(\frac{d}{q}-s)}\,\left(\log(1+ N)\right)^\gamma.
\end{eqnarray*}
If we choose $N$ to be the largest integer satisfying
$$
N \le \frac{1}{s-d/q} \log\left(1+ \|\theta\|_{B^{s}_{q,\infty}}\right),
$$
we then obtain the desired result in Proposition \ref{LogSob}.
\end{proof}

\vskip .1in
\section{Global regularity for (\ref{general}) with $P(\Lambda) = \left(\log(1+\log(1 -\Delta))\right)^\gamma$}
\label{APbd}

This section establishes the global existence and uniqueness of solutions to (\ref{general}) with $P(\Lambda) = \left(\log(1+\log(1 -\Delta))\right)^\gamma$. The divergence-free condition on the velocity field $u$ is not necessary if we are willing to assume that $\theta$ is bounded in $L^1\cap L^\infty$ for all time. Of course when $u$ is indeed divergence-free, the bound is then a trivial consequence. In the 2D case, this general theorem reduces to Theorem \ref{loglog-global} stated in the introduction.

\begin{thm} \label{AP}
Consider the active scalar equation (\ref{general}) with
$$
P(\Lambda) = \left(\log(1+\log(1 -\Delta))\right)^\gamma, \quad 0\le \gamma\le 1.
$$
Assume that the initial data $\theta_0$ satisfies
$$
\theta_0 \in X\equiv L^1(\mathbb{R}^d)\cap L^\infty(\mathbb{R}^d) \cap B^s_{q,\infty}(\mathbb{R}^d),
$$
with
$$
 d < q \le \infty \quad \mbox{and} \quad s>1.
$$
Assume either $u$ is divergence-free or $\theta$ is bounded in $L^1\cap L^\infty$ for all time. Then (\ref{general}) has a unique global in time solution $\theta$ that satisfies
$$
\theta\in L^\infty([0,\infty); B^s_{q,\infty}(\mathbb{R}^d))\quad\mbox{and}\quad u\in L^\infty([0,\infty); B^{1+s'}_{q,\infty}(\mathbb{R}^d))
$$
for any $s'<s$.
\end{thm}

\vskip .1in
\begin{proof}
The proof consists of two main components. The first component derives a global {\it a priori} bound while the second constructs a unique local in time solution through the method of successive approximation.

\vskip .1in
We start with the part on the global {\it a priori} bound. This part is further divided into two steps. The first step shows that for any $d/q<\sigma <1$ and any $T>0$,
$$
\|\theta(t)\|_{B^\sigma_{q,\infty}} \le C(T, \|\theta_0\|_X), \qquad t\le T
$$
and the second step establishes the global bound in $B^{\sigma_1}_{q,\infty}$ for some $\sigma_1>1$. A finite number of iterations then yields the global bound in $B^s_{q,\infty}$.

\vskip .1in
When $u$ is divergence-free, $\theta_0\in L^1 \cap L^\infty$ implies that the corresponding solution $\theta$ of (\ref{loglog-euler}) satisfies the {\it a priori} bound
\begin{equation} \label{mmm}
\|\theta(\cdot,t)\|_{L^1\cap L^\infty} \le \|\theta_0\|_{L^1\cap L^\infty},
\quad t \ge 0.
\end{equation}
When $u$ is not divergence-free, we assume that (\ref{mmm}) holds. Of course, the bound does not have to be $\|\theta_0\|_{L^1\cap L^\infty}$.
In the rest of the proof, we can completely avoid using the divergence-free condition on $u$. This explains why the divergence-free condition is not used in the estimates.

\vskip .1in
Let $j\ge -1$ be an integer. Applying $\Delta_j$ to (\ref{loglog-euler}) and following a standard decomposition, we have
\begin{equation}\label{base1}
\partial_t \Delta_j \theta = J_1 + J_2 + J_3 +J_4 +J_5
\end{equation}
where
\begin{eqnarray}
J_{1} &=& - \sum_{|j-k|\le 2}
[\Delta_j, S_{k-1}(u)\cdot\nabla] \Delta_k \theta, \nonumber\\
J_{2} &=& - \sum_{|j-k|\le 2}
(S_{k-1}(u) - S_j(u)) \cdot \nabla \Delta_j\Delta_k \theta, \nonumber\\
J_3   &=&  - S_j(u) \cdot\nabla \Delta_j\theta, \nonumber\\
J_{4} &=& - \sum_{|j-k|\le 2}
\Delta_j (\Delta_k u \cdot \nabla S_{k-1}
(\theta)), \nonumber\\
J_{5} &=& -\sum_{k\ge j-1}\Delta_j (\Delta_k u\cdot\nabla
\widetilde{\Delta}_k \theta) \nonumber
\end{eqnarray}
with $\widetilde{\Delta}_k = \Delta_{k-1} + \Delta_k + \Delta_{k+1}$.
Multiplying (\ref{base1}) by $\Delta_j\theta |\Delta_j \theta|^{q-2}$, integrating in space, integrating by part in the term associated with $J_3$, and applying H\"{o}lder's inequality, we have
\begin{equation}\label{root1}
\frac{d}{dt} \|\Delta_j \theta\|_{L^q} \le \|J_1\|_{L^q} + \|J_2\|_{L^q} + \|\widetilde{J_3}\|_{L^q} + \|J_4\|_{L^q} + \|J_5\|_{L^q}.
\end{equation}
By a standard commutator estimate,
$$
\|J_1\|_{L^q} \le C  \sum_{|j-k|\le 2} \|\nabla S_{k-1} u\|_{L^\infty} \|\Delta_k \theta\|_{L^q}.
$$
By H\"{o}lder's and Bernstein's inequalities,
$$
\|J_2\|_{L^q} \le C\, \|\nabla \widetilde{\Delta}_j u\|_{L^\infty} \, \|\Delta_j \theta\|_{L^q}.
$$
After integration by parts, the term $J_3$ leads to a term $\widetilde{J_3} =
\frac{1}{q}\left(\nabla\cdot S_j u\right)\Delta_j\theta$, and so
$$
\|{\widetilde{J_3}}\|_{L^q} \le C\,\|\nabla\cdot S_j u\|_{L^\infty} \, \|\Delta_j \theta\|_{L^q}.
$$
For $J_4$ and $J_5$, we have
\begin{eqnarray*}
\|J_4\|_{L^q} &\le&  \sum_{|j-k|\le 2} \|\Delta_k u\|_{L^\infty} \, \|\nabla S_{k-1} \theta\|_{L^q} \\
&\le& C\, \sum_{|j-k|\le 2} \|\nabla \Delta_k u\|_{L^\infty} \, \sum_{m\le k-1} 2^{m-k} \|\Delta_m\theta\|_{L^q},\\
\|J_5\|_{L^q} &\le& C \,\sum_{k\ge j-1} \,\|\Delta_k u\|_{L^\infty} \| \widetilde{\Delta}_k \nabla \theta\|_{L^q} \\
&\le& C\, \sum_{k\ge j-1} \|\nabla \Delta_k u\|_{L^\infty}\, \|\widetilde{\Delta}_k \theta\|_{L^q}.
\end{eqnarray*}
By Proposition \ref{LogSob}, for any $\sigma \in \mathbb{R}$,
\begin{eqnarray}
\|J_1\|_{L^q} &\le&  C\, \sum_{|j-k|\le 2} \|\nabla u\|_{L^\infty}  2^{- \sigma (k+1)} \, 2^{\sigma(k+1) } \|\Delta_k \theta\|_{L^q} \label{j1q} \\
 &\le&  C\, 2^{-\sigma (j+1)}\,\|\theta\|_{B^\sigma_{q,\infty}}\, \|\nabla u\|_{L^\infty}\, \, \sum_{|j-k|\le 2} 2^{\sigma (j-k)}\\
 &\le&  C\, 2^{-\sigma (j+1)}\,\|\theta\|_{B^\sigma_{q,\infty}}\, \|\nabla u\|_{L^\infty},
\end{eqnarray}
where $C$ is a constant depending on $\sigma$ only. It is clear that $\|J_2\|_{L^q}$  and $\|\widetilde{J_3}\|_{L^q}$ obey the same bound.  For any $\sigma<1$, we have
\begin{eqnarray*}
\|J_4\|_{L^q} &\le& C\, \|\nabla u\|_{L^\infty}\,  \sum_{|j-k|\le 2} \, \sum_{m< k-1} 2^{m-k}\, 2^{-\sigma(m+1)} \,2^{\sigma(m+1)} \,\|\Delta_m\theta\|_{L^q} \\
&\le& C\, \|\nabla u\|_{L^\infty}\,  \|\theta\|_{B^\sigma_{q,\infty}}\, \sum_{|j-k|\le 2} \, \sum_{m < k-1} 2^{m-k}\, 2^{-\sigma(m+1)}\\
&=& C\, 2^{-\sigma (j+1)}\,\|\theta\|_{B^\sigma_{q,\infty}}\, \|\nabla u\|_{L^\infty}\,\sum_{|j-k|\le 2} 2^{\sigma(j-k)} \sum_{m < k-1} 2^{(m-k)(1-\sigma)}\\
&\le& C\, 2^{-\sigma (j+1)}\,\|\theta\|_{B^\sigma_{q,\infty}}\, \|\nabla u\|_{L^\infty}.
\end{eqnarray*}
where $C$ is a constant depending on $\sigma$ only and the condition $\sigma<1$ is used to guarantee that $(m-k)(1-\sigma)<0$. For any $\sigma>0$,
\begin{eqnarray*}
\|J_5\|_{L^q} &\le& C \, \|\nabla u\|_{L^\infty}\, 2^{-\sigma(j+1)}\, \sum_{k\ge j-1}   2^{\sigma(j-k)}\, 2^{\sigma(k+1)}\, \|\widetilde{\Delta}_k \theta\|_{L^q} \\
&\le& C\, 2^{-\sigma (j+1)}\,\|\theta\|_{B^\sigma_{q,\infty}}\, \|\nabla u\|_{L^\infty}.
\end{eqnarray*}
Collecting these estimates, we obtain, for any $0<\sigma<1$,
\begin{eqnarray*}
\frac{d}{dt} \|\Delta_j \theta\|_{L^q} \le C\, 2^{-\sigma (j+1)}\,\|\theta\|_{B^\sigma_{q,\infty}}\, \|\nabla u\|_{L^\infty}.
\end{eqnarray*}
Integrating in time yields
$$
\|\theta(t)\|_{B^\sigma_{q,\infty}} \le \|\theta_0\|_{B^\sigma_{q,\infty}} + C\, \int_0^t \|\theta(\tau)\|_{B^\sigma_{q,\infty}}\, \|\nabla u(\tau)\|_{L^\infty}\,d\tau.
$$
Invoking the extrapolation inequality in Proposition \ref{LogSob}, we obtain, for $d/q<\sigma<1$,
\begin{eqnarray}
\|\theta(t)\|_{B^\sigma_{q,\infty}} &\le& \|\theta_0\|_{B^\sigma_{q,\infty}} + C\,\int_0^t \|\theta(\tau)\|_{B^\sigma_{q,\infty}}\,
 \Big[\|\theta\|_{L^1\cap L^\infty} + (1+\|\theta\|_{L^\infty})\,  \nonumber\\
&& \qquad\times \log(1+ \|\theta\|_{B^\sigma_{q,\infty}})\,\left(\log\left(1+\log(1+\|\theta\|_{B^\sigma_{q,\infty}})\right)\right)^\gamma  \Big]\,d\tau. \nonumber
\end{eqnarray}
It then follows from Gronwall's inequality that, for any $T>0$,
$$
\|\theta(t)\|_{B^\sigma_{q,\infty}} \le C(T, \|\theta_0\|_X), \qquad t\le T.
$$

\vskip .1in
We now continue with the second step. Since $d<q\le \infty$, we can choose $\sigma$ satisfying
$$
\frac{d}{q}<\sigma<1, \quad \sigma +1 -\frac{d}{q} > 1
$$
and then set $\sigma_1$ satisfying
$$
1< \sigma_1 < \sigma + 1 -\frac{d}{q}.
$$
This step establishes the global bound for $\|\theta\|_{B^{\sigma_1}_{q,\infty}}$.  $J_1$, $J_2$ and $J_3$ and $J_5$ can be bounded the same way as before, namely
$$
\|J_1\|_{L^q}, \,\|J_2\|_{L^q}, \,\|\widetilde{J_3}\|_{L^q}, \,\|J_5\|_{L^q} \le C\,  2^{-\sigma_1 (j+1)}\,\|\theta\|_{B^{\sigma_1}_{q,\infty}}\, \|\nabla u\|_{L^\infty}.
$$
$\|J_4\|_{L^q}$ is estimated differently and bounded by the global bound in the first step. We start with the bound
$$
\|J_4\|_{L^q} \le C\, \sum_{|j-k|\le 2} \|\nabla \Delta_k u\|_{L^\infty} \, \sum_{m< k-1} 2^{m-k} \|\Delta_m\theta\|_{L^q}.
$$
By Bernstein's inequality and Proposition \ref{important}, we have
\begin{eqnarray*}
\|\nabla \Delta_k u\|_{L^\infty} &\le& 2^{\frac{d k}{q}}\, \|\nabla \Delta_k u\|_{L^q} \\
&\le& 2^{\frac{d k}{q}}\, \left(\log(2 + k)\right)^\gamma  \|\Delta_k \theta\|_{L^q}.
\end{eqnarray*}
Clearly,
\begin{eqnarray*}
\sum_{m < k-1} 2^{m-k}\, \|\Delta_m \theta\|_{L^q} &=& 2^{-\sigma k} \sum_{m< k-1} 2^{(m-k)(1-\sigma)} 2^{\sigma m}\,\|\Delta_m \theta\|_{L^q} \\
&\le& C\, 2^{-\sigma k} \|\theta\|_{B^\sigma_{q,\infty}}.
\end{eqnarray*}
Therefore,
\begin{eqnarray*}
\|J_4\|_{L^q} &\le&   C\, \sum_{|j-k|\le 2} 2^{\frac{d k}{q}}\, \left(\log(2 + k)\right)^\gamma  \|\Delta_k \theta\|_{L^q} 2^{-\sigma k} \|\theta\|_{B^\sigma_{q,\infty}}\\
&=& C\, 2^{-\sigma_1(j+1)} \,\|\theta\|_{B^\sigma_{q,\infty}}\,\sum_{|j-k|\le 2} 2^{\sigma_1(j-k)} \left(\log(2 + k)\right)^\gamma \, 2^{(\sigma_1+\frac{d}{q}-\sigma) k}\|\Delta_k \theta\|_{L^q} \\
&=& C\, 2^{-\sigma_1(j+1)} \,\|\theta\|_{B^\sigma_{q,\infty}}\,
\|\theta\|_{B^{\sigma_2}_{q,\infty}} \, \sum_{|j-k|\le 2} 2^{\sigma_1(j-k)} \left(\log(2 + k)\right)^\gamma \, 2^{(\sigma_1+\frac{d}{q}-\sigma-\sigma_2) k}
\end{eqnarray*}
where $\sigma_2<1$ is chosen very close to $1$ and satisfies
$$
\sigma_1+\frac{2}{q}-\sigma-\sigma_2<0.
$$
Then, by the global bound in the first step,
$$
\|J_4\|_{L^q}  \le C\, 2^{-\sigma_1(j+1)} \,\|\theta\|_{B^\sigma_{q,\infty}}\,
\|\theta\|_{B^{\sigma_2}_{q,\infty}} \le C(T, \|\theta_0\|_X) \,  2^{-\sigma_1(j+1)}.
$$
Collecting the estimates in this step, we have
$$
\frac{d}{dt} \|\Delta_j \theta\|_{L^q} \le C\, 2^{-\sigma_1 (j+1)}\,\|\theta\|_{B^{\sigma_1}_{q,\infty}}\, \|\nabla u\|_{L^\infty} + C(T, \|\theta_0\|_X) \,  2^{-\sigma_1(j+1)}.
$$
By Proposition \ref{LogSob}, for any $d/q< \sigma<1$,
\begin{eqnarray*}
\|\nabla u\|_{L^\infty} &\le& \|\theta\|_{L^1\cap L^\infty} + (1+\|\theta\|_{L^\infty})\\
&& \quad \times \log(1+ \|\theta\|_{B^\sigma_{q,\infty}})\,
\left(\log\left(1+\log(1+\|\theta\|_{B^\sigma_{q,\infty}})\right)\right)^\gamma\\
&\le& C(T, \|\theta_0\|_X).
\end{eqnarray*}
Therefore,
$$
\|\theta(t)\|_{B^{\sigma_1}_{q,\infty}}\, \le \|\theta_0\|_{B^{\sigma_1}_{q,\infty}}\, + C(T, \|\theta_0\|_X)\left(1+ \int_0^t \|\theta(\tau)\|_{B^{\sigma_1}_{q,\infty}}\,d\tau\right).
$$
Gronwall's inequality then yields the global bound $\|\theta(t)\|_{B^{\sigma_1}_{q,\infty}} \le C(T, \|\theta_0\|_X)$. If $s>\sigma_1$, we can repeat this step to achieve the desired regularity.

\vskip .1in
We now describe the process of constructing a local solution of (\ref{general}).
The solution is constructed through the method of successive approximation.
Consider a successive approximation sequence $\{\theta^{(n)}\}$ satisfying
\begin{equation}\label{succ}
\left\{
\begin{array}{l}
\theta^{(1)} = S_2 \theta_0, \\ \\
u^{(n)} = (u^{(n)}_j), \quad u^{(n)}_j = \mathcal{R}_l \Lambda^{-1} P(\Lambda) \theta^{(n)},\\ \\
\partial_t \theta^{(n+1)} + u^{(n)} \cdot\nabla \theta^{(n+1)} = 0,\\ \\
\theta^{(n+1)}(x,0) = S_{n+2} \theta_0,
\end{array}
\right.
\end{equation}
where $P(\Lambda) = \left(\log(1+\log(1 -\Delta))\right)^\gamma$. In order to show that $\{\theta^{(n)}\}$ converges to a solution of (\ref{general}) , it suffices to prove the following properties of $\{\theta^{(n)}\}$:
\begin{enumerate}
\item There exists $T_1>0$ such that $\theta^{(n)}$ is bounded uniformly in $B^{s}_{q,\infty}$ for any $t\in[0,T]$, namely
$$
\|\theta^{(n)}(\cdot,t)\|_{B^{s}_{q,\infty}} \le C_1 \|\theta_0\|_{X}, \quad t\in [0,T_1],
$$
where $C_1$ is a constant independent of $n$.
\item There exists $T_2>0$ such that $\eta^{(n+1)} = \theta^{(n+1)}- \theta^{(n)}$ is a Cauchy sequence in $B^{s-1}_{q,\infty}$,
$$
\|\eta^{(n)}(\cdot,t)\|_{B^{s-1}_{q,\infty}} \le C_2\, 2^{-n}, \quad t\in [0, T_2],
$$
where $C_2$ is independent of $n$ and depends on $T_2$ and $\|\theta_0\|_X$ only.
\end{enumerate}
These two properties are established by following the ideas of the previous part and we omit the details. Let $T= \min\{T_1, T_2\}$. We conclude from these two properties that there exists $\theta$ satisfying
\begin{eqnarray*}
&& \theta(\cdot, t) \in B^{s}_{q,\infty} \quad \mbox{for}\quad  0\le t\le T, \\
&&  \theta^{(n)}(\cdot, t) \rightharpoonup \theta(\cdot, t) \quad\mbox{in }\quad B^{s}_{q,\infty},\\
&& \theta^{(n)}(\cdot, t) \rightarrow \theta(\cdot, t) \quad\mbox{in }\quad B^{s-1}_{q,\infty}.
\end{eqnarray*}
Due to the interpolation inequality, for any $s-1\le \widetilde{s}\le s$,
$$
\|f\|_{B^{\widetilde{s}}_{q,\infty}} \le C \, \|f\|^{s-\widetilde{s}}_{B^{s-1}_{q,\infty}}\, \|f\|^{\widetilde{s}+1-s}_{B^{s}_{q,\infty}},
$$
we deduce that
\begin{equation}\label{thnto}
\theta^{(n)}(\cdot, t) \rightarrow \theta(\cdot, t) \quad\mbox{in }\quad B^{\widetilde{s}}_{q,\infty}.
\end{equation}
In addition, by the relation $u^{(n)}_k = \mathcal{R}_l \Lambda^{-1} P(\Lambda)\, \theta^{(n)}$ and Proposition \ref{important}, we can easily check that
$$
\nabla u^{(n)}, \,\,\nabla u(\cdot, t)  \in B^{s_1}_{q,\infty} \quad\mbox{for any}\quad s_1 < s.
$$
In order to pass the limit in the nonlinear term, we write
$$
u^{(n)}\cdot\nabla \theta^{(n+1)} - u\cdot\nabla \theta = u^{(n)}\cdot \nabla (\theta^{(n+1)}-\theta) + (u^{(n)}-u) \cdot \nabla\theta.
$$
We can show that, for any $\sigma <s-1$,
\begin{equation} \label{conv}
u^{(n)}\cdot \nabla (\theta^{(n+1)}-\theta) \to 0, \quad (u^{(n)}-u) \cdot \nabla\theta \to 0 \quad \mbox{in} \quad B^{\sigma}_{q,\infty},
\end{equation}
as $n\to \infty$.
Again they can be proven by following the ideas in the first part of this proof. Finally the uniqueness can be established by estimating the difference of any two solutions in $B^{s-1}_{q,\infty}$. A similar argument as in the proof of $\|\eta^{(n)}(\cdot,t)\|_{B^{s-1}_{q,\infty}} \le C_2\, 2^{-n}$ yields the conclusion that the difference must be zero. This completes the proof of Theorem \ref{AP}.
\end{proof}

\vskip .3in
\section{Generalized Inviscid SQG equation}
\label{GISQG}

This section is devoted to the generalized inviscid SQG equation
\begin{equation} \label{SQG_beta}
\left\{
\begin{array}{l}
\pp_t \theta + (u\cdot\nabla) \theta =0, \quad x\in \mathbb{R}^2, \,\,t>0, \\
u = \nabla^\perp \psi,\quad -\Lambda^{2-\beta} \psi = \theta,\quad x\in \mathbb{R}^2, \,\,t>0,
\end{array}
\right.
\end{equation}
where $0\le \beta\le 1$ is a parameter. (\ref{SQG_beta}) with $\beta=0$ becomes the 2D Euler vorticity equation while (\ref{SQG_beta}) with $\beta=1$ is the SQG equation. Except in the case when $\beta=0$, the global regularity issue for (\ref{SQG_beta}) remains open. This section presents a regularity criterion in terms of the  norm of $\theta$ in the H\"{o}lder space $C^\beta(\mathbb{R}^2)$, which directly relates the regularity of $\theta$ to the parameter $\beta$. The precise conclusion has been stated in Theorem \ref{crit10} and we reproduce it here.

\vskip .1in
\begin{thm}\label{crit1}
Consider (\ref{SQG_beta}) with $0\le \beta\le 1$. Let $\theta$ be a solution of (\ref{SQG_beta}) corresponding to the data $\theta_0 \in C^\sigma(\mathbb{R}^2)\cap L^q(\mathbb{R}^2)$ with $\sigma>1$ and $q>1$. Let $T>0$. If $\theta$ satisfies
\begin{equation}\label{cdelta}
\int_0^{T}\|\theta(\cdot, t)\|_{C^\beta(\mathbb{R}^2)} \,dt < \infty,
\end{equation}
then $\theta$ remains in $C^\sigma(\mathbb{R}^2)\cap L^q(\mathbb{R}^2)$ on the time interval $[0,T]$.
\end{thm}

\vskip .1in
Some special consequences of this theorem are given in the following remark.
\begin{rem}
In the special case when $\beta=0$, Theorem \ref{crit1} re-establishes the global regularity for the 2D Euler equation. In the special case when $\beta=1$, (\ref{SQG_beta}) becomes the inviscid SQG equation and Theorem \ref{crit1} reduces to a regularity criterion of \cite{CMT} for the SQG equation.
\end{rem}

\vskip .1in
To prove Theorem \ref{crit1}, we first establish two propositions. The first one bounds the back-to-labels map (the inverse map of the particle trajectory) in terms of the symmetric part of $\nabla u$. The second proposition is a logarithmic H\"{o}lder space
inequality.

\vskip .1in
Let $X(a,t)$ be the particle trajectory determined by the velocity $u$, namely
\begin{equation}\label{part}
\left\{
\begin{array}{l}
\frac{d X(a,t)}{dt} =u(X(a,t),t),\\
X(a,0)=a.
\end{array}
\right.
\end{equation}
Let $A(x,t)$ be the back-to-labels map or the inverse map of $X$. Then
\begin{equation}\label{axid}
A(X(a,t),t) =a \quad\mbox{for any $a\in \mathbb{R}^2$}.
\end{equation}
Let $S$ denote the symmetric part of $\nabla u$, namely
\begin{equation}\label{Sdef}
S= \frac12 \left(\nabla u + (\nabla u)^T\right),
\end{equation}
where $(\nabla u)^T$ denotes the transpose of $\nabla u$. The following proposition bounds $\nabla_x A$ in terms of $S$.

\begin{prop} \label{nainfty}
Let $u$ be a velocity field and let $S$ be the strain tensor as defined in (\ref{Sdef}).  Let $A$ be the back-to-labels map. Then,
$$
\|\nabla_x A(\cdot,t)\|_{L^\infty}  \le \exp \left(\int_0^t \|S(\cdot,\tau)\|_{L^\infty} \,d\tau\right).
$$
\end{prop}

\vskip .1in
The second proposition bounds the $L^\infty$-norm of $S$ in terms of the logarithm of the H\"{o}lder-norm of $\theta$.
\begin{prop} \label{sinfty}
Let $0\le \beta\le 1$. Assume that $u$ and $\theta$ are related by
\begin{equation} \label{uth}
u = -\nabla^\perp \Lambda^{-2+\beta} \theta
\end{equation}
If $\theta\in C^\sigma(\mathbb{R}^2) \cap L^q(\mathbb{R}^2)$ with $\sigma>\beta$ and $q>1$,
\begin{equation} \label{logs}
\|S\|_{L^\infty} \le C_1 \|\theta\|_{C^\beta} \,\ln (1+ \|\theta\|_{C^\sigma}) + C_2 \|\theta\|_{L^q},
\end{equation}
where $C_1$ and $C_2$ are constants depending  on $\beta$, $\sigma$ and $q$ only.
\end{prop}

The rest of this section is arranged as follows. We prove Theorem \ref{crit1} first and then provide the proofs of Propositions \ref{nainfty} and \ref{sinfty}.

\vskip .1in
\begin{proof}[Proof of Theorem \ref{crit1}] Let $X$ be the particle trajectory as defined in (\ref{part}) and $A(x,t)$ be the back-to-labels map. The first equation in (\ref{SQG_beta}) implies that $\theta$ is conserved along the particle trajectory,
$$
\theta(x,t) = \theta_0(A(x,t)), \quad x \in \mathbb{R}^2, \, t\ge 0.
$$
Therefore, for any $\sigma\le 1$,
$$
\|\theta(\cdot,t)\|_{C^\sigma} =\sup_{x\not = y}\frac{|\theta(x,t)-\theta(y,t)|}{|x-y|^\sigma} \le \|\theta_0\|_{C^\sigma}\, \|\nabla_x A(\cdot,t)\|_{L^\infty}^\sigma.
$$
By Proposition \ref{nainfty}, we have
$$
\|\theta(\cdot,t)\|_{C^\sigma} \le \|\theta_0\|_{C^\sigma}\, \exp\left(\sigma\,\int_0^t \|S(\cdot,\tau)\|_{L^\infty} \,d\tau\right).
$$
Therefore,
\begin{equation} \label{g1}
\ln(1+\|\theta(\cdot,t)\|_{C^\sigma}) \le \ln(1+\|\theta_0\|_{C^\sigma}) + \sigma \int_0^t \|S(\cdot,\tau)\|_{L^\infty} \,d\tau.
\end{equation}
According to Proposition \ref{sinfty},
\begin{equation} \label{g2}
\int_0^t \|S(\cdot,\tau)\|_{L^\infty} \,d\tau \le C_1 \int_0^t \|\theta(\cdot,\tau)\|_{C^\beta} \,\ln (1+ \|\theta(\cdot,\tau)\|_{C^\sigma})\,d\tau + C_2 t\,\|\theta_0\|_{L^q}.
\end{equation}
Combining (\ref{g1}) and (\ref{g2}) and applying Gronwall's inequality yield
$$
\ln(1+\|\theta(\cdot,t)\|_{C^\sigma}) \le C\,\ln(1+\|\theta_0\|_{C^\sigma} + \|\theta_0\|_{L^q}) \exp\left(C\,\int_0^t \|\theta(\cdot,\tau)\|_{C^\beta}\,d\tau\right).
$$
In particular, taking $\sigma=1$ yields a bound for $\|\nabla \theta\|_{L^\infty}$.
The desired regularity $\theta\in C^\sigma$ with $\sigma>1$ then follows easily from the bound for $\|\nabla \theta\|_{L^\infty}$. This completes the proof of Theorem \ref{crit1}.
\end{proof}

\vskip .1in
\begin{proof}[Proof of Proposition \ref{nainfty}]
Differentiating the identity in (\ref{axid}) with respect to $t$, we obtain the equation for $A$, $$
\pp_t A + u\cdot\nabla A =0.
$$
Taking the gradient with respect to $x$, we find
$$
\pp_t (\nabla_x A) + u\cdot\nabla (\nabla_x A) = \nabla u\,  (\nabla_x A).
$$
Taking (Euclidian) inner product of this equation with $\nabla_x A$,
we find
$$
\frac{1}{2} \frac{D}{Dt} |\nabla_x A(x,t)|^2 =- \nabla u\, (\nabla_x
A))\cdot (\nabla_x A) .
$$
Adopting the Einstein summation convention, we have
$$
(\nabla u\,  (\nabla_x A))\cdot (\nabla_x A)  = \partial_{x_k} u_j\, \partial_{x_j}A_i \,\partial_{x_k} A_i = \partial_{x_j} u_k\, \partial_{x_k}A_i \,\partial_{x_j} A_i
$$
and thus
$$
(\nabla u\,  (\nabla_x A))\cdot (\nabla_x A) = ((\nabla u)^T\,  (\nabla_x A))\cdot (\nabla_x A)=  (S (\nabla_x A))\cdot (\nabla_x A).
$$
Therefore
$$
\frac12\frac{D}{Dt} |\nabla_x A|^2 \le |S(x,t)| |\nabla_x A|^2\leq
\|S (\cdot, t)\|_{L^\infty} |\nabla_x A|^2,
$$
and integrating along the particle trajectory we obtain
$$
|\nabla_x A(X(a,t),t)|\leq \exp\left(\int_0 ^t\|S (\cdot, \tau
)\|_{L^\infty}d\tau\right).
$$
Proposition \ref{nainfty} follows from this immediately, taking
supremum over $a\in \mathbb{ R}^2$.
\end{proof}

\vskip .1in
\begin{proof}[Proof of Proposition \ref{sinfty}]

The proof is divided into two cases: $\beta<1$ and $\beta=1$.
The case $\beta=1$ requires that $\sigma>1$ and is handled differently from
the case $\beta<1$.

\vskip .1in
We first deal with the case when $\beta<1$.  Invoking the Riesz potential for the
operator $\Lambda^{-2+\beta}$, the relation in (\ref{uth}) can be rewritten
$$
u(x) = C_\beta\,\int \nabla^\perp \left(\frac{1}{|x-y|^\beta}\right)\,\theta(y) \,dy
=\int K_\beta(x-y)\, \theta(y) \,dy
$$
with
$$
K_\beta(x) = C_\beta \, \frac{(-x_2, x_1)^T}{|x|^{2+\beta}},
$$
where $C_\beta$ is a  constant depending on $\beta$ only.
$\nabla u$ can be written as
$$
\nabla u (x) = \mbox{p.v.} \int \nabla_x K(x-y)\, \theta(y) \,dy,
$$
where $\mbox{p.v.}$ denotes the principal value and $\nabla_x K(x)$ can be explicitly written as
$$
\nabla_x K(x) =C_\beta\, \frac{1}{|x|^{4+\beta}}
\left(
\begin{array}{cc}
x_1 x_2  & x_2^2 \\ -x_1^2 & -x_1x_2
\end{array}
\right)
\,+\,C_\beta \frac{1}{|x|^{2+\beta}}
\left(
\begin{array}{cc}
0 & -1\\ 1 & 0
\end{array}
\right).
$$
Therefore the symmetric part of $\nabla u$ can be written as
$$
S(x) = \mbox{p.v.} \int \Gamma(x-y) \,\theta(y)\,dy
$$
where
$$
\Gamma(x) =C_\beta \frac{1}{|x|^{4+\beta}}
\left(
\begin{array}{cc}
2 x_1 x_2  & x_2^2-x_1^2 \\ x_2^2-x_1^2 & -2 x_1x_2
\end{array}
\right).
$$
The property that $\Gamma(x)$ is homogenous of degree $-(2+\beta)$ and has zero mean on the unit circle is useful in the following estimate of $S$.

\vskip .1in
Let $\chi(x)$ be a standard smooth cutoff function with $\chi(x)=1$ for $|x| \le \frac12$ and $\chi(x)=0$ for $|x|\ge 1$. Let $0<\rho \le R$. We divide $S$ into three parts,
$$
S(x,t) = L_1 + L_2 + L_3,
$$
where
\begin{eqnarray*}
L_1 &=& \int \chi\left(\frac{|x-y|}{\rho}\right) \Gamma(x-y) \,(\theta(y)-\theta(x))\,dy,
\\
L_2 &=& \int_{|x-y|\le R} \left(1-\chi\left(\frac{|x-y|}{\rho}\right)\right) \Gamma(x-y) \,(\theta(y)-\theta(x))\,dy,
\\
L_3 &=& \int_{|x-y|> R} \Gamma(x-y) \,\theta(y)\,dy.
\end{eqnarray*}
Since $\sigma>\beta$,
\begin{eqnarray*}
|L_1| &\le&  C_\beta \,\|\theta\|_{C^\sigma}\,\int_{|x-y| \le \rho} \frac{1}{|x-y|^{2+\beta-\sigma}} \,dy
\\
&=& C_\beta \,\|\theta\|_{C^\sigma}\, \rho^{\sigma-\beta}.
\end{eqnarray*}
$L_2$ can be bounded as follows.
\begin{eqnarray*}
|L_2| &\le&  C_\beta\,\|\theta\|_{C^\beta}\, \int_{\frac{\rho}{2} \le |x-y| \le R} \frac{1}{|x-y|^{2}} \,dy \\
&=&  C_\beta\, \|\theta\|_{C^\beta}\,\ln\left(\frac{2R}{\rho}\right).
\end{eqnarray*}
By H\"{o}lder's inequality,
$$
|L_3| \le C_{\beta,q}\, R^{-1-\beta} \|\theta\|_{L^q}
$$
Setting $\rho = \ln (1 + \|\theta\|_{C^\sigma})$ and $R=1$ yields (\ref{logs}).

\vskip .1in
We now turn to the case when $\beta=1$. This case corresponds to the SQG equation. Then $\sigma>\beta =1$. It follows from the relation in (\ref{uth}) that
$$
\nabla u(x) = \mbox{p.v.} \int \hat{y} \otimes \nabla \theta(x+y) \frac{dy}{|y|^2}
$$
where $\hat{y}$ denotes the unit vector in the direction of $y$ and $a\otimes b$ denotes the tensor product of two vectors $a$ and $b$. Therefore,
$$
S(x) = \mbox{p.v.} \int \frac12\left(\hat{y} \otimes \nabla \theta(x+y) + \nabla \theta(x+y)\otimes \hat{y}\right) \frac{dy}{|y|^2}.
$$
The difference between this representation and the one in the case $\beta<1$ is that
this formula involves $\nabla \theta$ instead of just $\theta$. $\|S\|_{L^\infty}$ can
be bounded in a similar fashion as in the case $\beta<1$. In fact, we again use a
smooth cutoff function $\chi$ to decompose the integral into three parts and estimate
each one of them as we did previously. For example,
$$
L_1= \mbox{p.v.} \int \chi\left(\frac{|y|}{\rho}\right) \frac12\left(\hat{y} \otimes (\nabla \theta(x+y)-\nabla\theta(x)) + (\nabla \theta(x+y)-\nabla\theta(x))\otimes \hat{y}\right) \frac{dy}{|y|^2}
$$
can be bounded by
\begin{eqnarray*}
|L_1| &\le& \int_{|y|\le \rho} |\nabla \theta(x+y)-\nabla\theta(x)| \frac{dy}{|y|^2}
\\
&\le& \|\nabla \theta\|_{C^{\sigma-1}}\, \rho^{\sigma-1} \le \|\theta\|_{C^\sigma}\,\rho^{\sigma-1}.
\end{eqnarray*}
We omit details for the estimates of the other parts. Putting the estimates together yield the same bound as in the case $\beta<1$. This completes the proof of Propostion \ref{sinfty}.
\end{proof}

\vskip .4in
\appendix

\section{Besov spaces and related facts}
\label{Besov}

This appendix provides the definitions of $\Delta_j$, $S_j$ and inhomogeneous Besov spaces. Related useful facts such as the Bernstein inequality are also provided here. Materials presented in this appendix here can be found in several books and papers (see e.g. \cite{BL},\cite{Che} or \cite{Tr}).

\vskip .1in
Let ${\mathcal S}({\mathbf R}^d)$ and ${\mathcal S}'({\mathbf R}^d)$ denote the Schwartz class and tempered distributions, respectively. The partition of unity states that there exist two nonnegative radial functions $\psi, \phi \in {\mathcal S}$ such that
\begin{eqnarray*}
&& \mbox{supp}\, \psi \subset \,\,B\left(0,\frac{11}{12}\right), \quad \mbox{supp}\, \phi \subset \,\,A\left(0,\frac34, \frac{11}{6}\right),\\
&& \psi(\xi) + \sum_{j\ge 0} \phi_j(\xi) = 1 \quad \mbox{for}\quad \xi \in {\mathbf R}^d, \qquad \phi_j(\xi) = \phi(2^{-j}\,\xi),\\
&& \mbox{supp}\, \psi \cap \mbox{supp}\, \phi_j = \emptyset\quad \mbox{if}  \,\,j\ge 1,\\
&& \mbox{supp}\, \phi_j \cap \mbox{supp}\, \phi_k = \emptyset\quad \mbox{if}  \,|j-k|\ge 2,
\end{eqnarray*}
where $B(0,r)$ denotes the ball centered at the origin with radius $r$ and $A(0,r_1,r_2)$ the annulus centered at the origin with the inner radius $r_1$ and the outer radius $r_2$.

\vskip .1in
For any $f\in {\mathcal S}'$, set
\begin{eqnarray*}
&& \Delta_{-1} f = \mathcal{F}^{-1}\left(\psi(\xi) \mathcal{F}(f)\right) = \Psi \ast f, \\
&& \Delta_j f = \mathcal{F}^{-1}\left(\phi_j(\xi) \mathcal{F}(f)\right) = \Phi_j \ast f, \quad j=0,1,2,\cdots, \\
&& \Delta_j f =0\quad \mbox{for}\quad j\le -2,\\
&& S_j =\sum_{k=-1}^{j-1} \Delta_k \,\,\mbox{when}\quad j\ge 0,
\end{eqnarray*}
where we have used $\mathcal{F}$ and $\mathcal{F}^{-1}$ to denote the Fourier and inverse Fourier transforms. respectively. Clearly,
$$
\Psi = \mathcal{F}^{-1}(\psi), \quad \Phi_0 =\Phi =\mathcal{F}^{-1}(\phi), \quad \Phi_j(x) = \mathcal{F}^{-1}(\phi_j)(x) = 2^{jd}\,\Phi(2^jx).
$$
In addition, we can write
$$
\mathcal{F}(S_j f) = \psi\left(\frac{\xi}{2^j}\right)\, \mathcal{F}(f).
$$

\vskip .1in With these notation at our disposal, we now provide the definition of the inhomogeneous Besov space.
\begin{define}
For $s \in {\mathbf R}$ and $1\le p,q \le \infty$, the inhomogeneous
Besov space $B_{p,q}^s$ is defined by
$$
B_{p,q}^s = \left\{f\in {\mathcal S}':\,\,\|f\|_{B_{p,q}^s}<\infty
\right\},
$$
where
\begin{equation}\label{norm2}
\|f||_{B^s_{p,q}} \equiv \left\{
\begin{array}{ll}
\displaystyle  \Big(\sum_{j=-1}^\infty \left(2^{js}\, \|\Delta_j
f\|_{L^p}\Big)^q \right)^{1/q},\quad & \mbox{if $q<\infty$},
\\
\displaystyle  \sup_{-1\le j<\infty} 2^{js} \, \|\Delta_j f\|_{L^p},
\quad &\mbox{if $q=\infty$}.
\end{array}
\right.
\end{equation}
\end{define}

\vskip .1in
The Besov spaces and the standard Sobolev spaces defined by
$$
W^s_p=
(1-\Delta)^{-s/2}L^p
$$
obey the simple facts stated in the following lemma (see \cite{BL}).
\begin{prop}
Assume that $s\in {\mathbf R}$ and $p,q\in [1,\infty]$.
\begin{enumerate}
\item[1)] If $s_1 \le s_2$, then $B^{s_2}_{p,q}
\subset B^{s_1}_{p,q}$,
\item[2)]  If $1\le q_1 \le q_2 \le \infty$, then $B^s_{p,q_1} \subset B^s_{p,q_2}$,
\item[3)]If $1\le p_1\le p_2\le \infty$, $1\le q_1, q_2
\le \infty$, and $s_1\ge s_2 + d(\frac{1}{p_1}-\frac{1}{p_2})$, then
$$B^{s_1}_{p_1,q_1}({\mathbf R}^d) \subset B^{s_2}_{p_2,q_2}({\mathbf
R}^d),$$
\item[4)] If $1<p<\infty$, then
$$
B^s_{p,\min(p,2)} \subset W^s_p \subset B^s_{p,\max(p,2)}.
$$
\end{enumerate}
\end{prop}

\vskip .1in
The following Bernstein type inequalities are very useful and have been used in the previous sections.
\begin{prop}\label{bern}
Let $\alpha\ge0$. Let $1\le p\le q\le \infty$.
\begin{enumerate}
\item[1)] If $f$ satisfies
$$
\mbox{supp}\, \widehat{f} \subset \{\xi\in {\mathbf R}^d: \,\, |\xi|
\le K 2^j \},
$$
for some integer $j$ and a constant $K>0$, then
$$
\max_{|\beta|=k}\|D^\beta f\|_{L^q({\mathbf R}^d)} \le C\, 2^{k j +
jd(\frac{1}{p}-\frac{1}{q})} \|f\|_{L^p({\mathbf R}^d)},
$$
$$
\|(-\Delta)^\alpha f\|_{L^q({\mathbf R}^d)} \le C\, 2^{2\alpha j +
jd(\frac{1}{p}-\frac{1}{q})} \|f\|_{L^p({\mathbf R}^d)}
$$
for some constant $C$ depending on $K$, $p$ and $q$ only.
\item[2)] If $f$ satisfies
$$
\mbox{supp}\, \widehat{f} \subset \{\xi\in {\mathbf R}^d: \,\, K_12^j
\le |\xi| \le K_2 2^j \}
$$
for some integer $j$ and constants $0<K_1\le K_2$, then
$$
C\, 2^{k j} \|f\|_{L^q({\mathbf R}^d)} \le \max_{|\beta|=k}\|D^\beta f\|_{L^q({\mathbf R}^d)} \le C\, 2^{k j +
jd(\frac{1}{p}-\frac{1}{q})} \|f\|_{L^p({\mathbf R}^d)},
$$
$$
C\, 2^{2\alpha j} \|f\|_{L^q({\mathbf R}^d)} \le \|(-\Delta)^\alpha
f\|_{L^q({\mathbf R}^d)} \le C\, 2^{2\alpha j +
jd(\frac{1}{p}-\frac{1}{q})} \|f\|_{L^p({\mathbf R}^d)},
$$
where the constants $C$ depend on $K_1$, $K_2$, $p$ and $q$ only.
\end{enumerate}
\end{prop}

\vskip .4in
\section*{Acknowledgements}
Chae's research was partially supported by NRF grant No.2006-0093854. Constantin's research was partially supported by NSF grant DMS 0804380. Wu's research was partially supported by NSF grant DMS 0907913. Wu thanks the Department of Mathematics at Sungkyunkwan University for its hospitality during his visit there, and thanks Professor Changxing Miao for discussions.

\end{document}